\documentclass[11pt,a4paper]{article}

\usepackage{amsmath}
\usepackage{amsthm}
\usepackage{amsfonts}
\usepackage{color}
\usepackage{setspace}
\usepackage{mathrsfs}
\usepackage{graphicx}
\usepackage{calc}
\usepackage{tikz}

\newtheorem{theorem}{Theorem}
\newtheorem{conjecture}{Conjecture}
\newtheorem{lemma}{Lemma}

\newtheorem{claim}{Claim}
\newtheorem{corollary}{Corollary}

\usetikzlibrary{decorations.markings}

\tikzstyle{vertex}=[circle, draw, inner sep=2pt, minimum size=6pt]
\tikzstyle{filledvertex}=[circle, draw, fill, inner sep=2pt, minimum size=6pt]
\tikzstyle{directed}=[postaction={decorate,
	decoration={markings,mark=at position 0.5 with {\arrow{stealth}}}
}]

\usepackage{geometry}
\begin{document}
\begin{spacing}{1.15}

\title{On the number of vertex-disjoint cycles in digraphs
\thanks{Supported by NSFC (Nos. 11601430 and 11671320) and China Postdoctoral Science Foundation (No. 2016M590969).}}

\author{\quad Yandong Bai $^{a,}$\thanks{Corresponding author. E-mail addresses: bai@nwpu.edu.cn (Y. Bai),  yannis@lri.fr (Y. Manoussakis).},
\quad Yannis Manoussakis $^{b}$\\[2mm]
\small $^{a}$ Department of Applied Mathematics, Northwestern Polytechnical University, \\
\small Xi'an 710129, China\\
\small $^{b}$ Laboratoire de Recherche en Informatique, Universit\'{e} de Paris-Sud, \\
\small Orsay cedex 91405, France}
\date{\today}
\maketitle
\begin{abstract}
Let $k$ be a positive integer.
Bermond and Thomassen conjectured in 1981 that every digraph with minimum outdegree at least $2k-1$ contains $k$ vertex-disjoint cycles.
It is famous as one of the one hundred unsolved problems selected in [Bondy, Murty, Graph Theory, Springer-Verlag London, 2008].
Lichiardopol, Por and Sereni proved in [SIAM J. Discrete Math. 23 (2) (2009) 979-992] that the above conjecture holds for $k=3$.

Let $g$ be the girth, i.e., the length of the shortest cycle, of a given digraph.
Bang-Jensen, Bessy and Thomass\'{e} conjectured in [J. Graph Theory 75 (3) (2014) 284-302] that
every digraph with girth $g$ and minimum outdegree at least $\frac{g}{g-1}k$ contains $k$ vertex-disjoint cycles.
Thomass\'{e} conjectured around 2005 that
every oriented graph (a digraph without 2-cycles) with girth $g$ and minimum outdegree at least $h$ contains a path of length $h(g-1)$,
where $h$ is a positive integer.

In this note,
we first present a new shorter proof of the Bermond-Thomassen conjecture for the case of $k=3$,
and then we disprove the conjecture proposed by Bang-Jensen, Bessy and Thomass\'{e}.
Finally, we disprove the even girth case of the conjecture proposed by Thomass\'{e}.
\medskip

\smallskip
\end{abstract}

\section{Introduction}

Throughout this note,
a cycle (path) in a digraph always means a {\em directed} cycle (path).
We use Bang-Jensen and Gutin \cite{BG2008} for terminology and notation not defined here.
Only finite and simple digraphs are considered.

The problem of finding vertex-disjoint cycles in digraphs has received extensive study in the past decades,
one can see, e.g., \cite{Alon1996,BK2011,BKM2014,BT1981,GT2011,HY2012,Lichiardopol2014,RRST1996,Thomassen1983,Thomassen1987}.
For a positive integer $k$,
denote by $f(k)$ the minimum integer such that every digraph
with minimum outdegree at least $f(k)$
contains $k$ vertex-disjoint cycles.
In view of the complete symmetric digraph on $2k-1$ vertices,
we have $f(k)\geq 2k-1$.
In 1981, Bermond and Thomassen \cite{BT1981} conjectured that the equality holds.

\begin{conjecture}[Bermond and Thomassen \cite{BT1981}]\label{conjecture: bermond-thomassen}
$f(k)=2k-1$.
\end{conjecture}

The Bermond-Thomassen conjecture is famous as one of the one hundred unsolved problems selected in Bondy and Murty \cite{BM2008}.
It is trivially true for $k=1$ and Thomassen \cite {Thomassen1983} proved it for $k=2$ in 1983.
Lichiardopol et al. \cite{LPS2009} made a major step by showing that it is true for $k=3$ in 2009.
For the general value $f(k)$,
the best known bound $f(k)\leq 64k$ was obtained by Alon \cite{Alon1996} via a probabilistic argument in 1996.
Besides, Bessy et al. \cite {BLS2010} in 2010,
Bang-Jensen et al. \cite{BBT2014} in 2014
and Bai et al. \cite{BLL2015} in 2015
verified it for regular tournaments, tournaments and bipartite tournaments, respectively.

In 2014,
Bang-Jensen et al. \cite{BBT2014} considered the existence of $k$ vertex-disjoint cycles
in terms of minimum outdegree and girth (length of the shortest cycle) of the digraph.
Let $f(k,g)$ be the minimum integer such that every digraph
with girth $g$ and minimum outdegree at least $f(k,g)$
contains $k$ vertex-disjoint cycles.
Clearly,
$$
f(k)=\max_{g}f(k,g)
$$
and thus $f(k,g)\leq f(k)$.
In view of the fact that the circular digraph with order $n=p(g-1)+1$
and minimum outdegree at least $p=\lfloor \frac{g}{g-1}k\rfloor$
contains no $k$ vertex-disjoint cycles,
Bang-Jensen et al. \cite{BBT2014} pointed out that $f(k,g)\geq \lceil \frac{g}{g-1}k\rceil$
and conjectured that the equality holds.

\begin{conjecture}[Bang-Jensen, Bessy and Thomass\'{e} \cite{BBT2014}]\label{conjecture: bang-jensen et al.}
$f(k,g)=\lceil \frac{g}{g-1}k\rceil$.
\end{conjecture}

An {\em oriented} graph is a digraph without 2-cycles.
Let $l(h,g)$ be the minimum integer such that
every oriented graph with girth $g$ and minimum outdegree at least $h$ contains a path of length $l(h,g)$.
Sullivan noted in \cite{Sullivan2006} that Thomass\'{e} made the following conjecture on $l(h,g)$.

\begin{conjecture}[Thomass\'{e}, see \cite{Sullivan2006}]\label{conj: path}
$l(h,g)\geq h(g-1)$.
\end{conjecture}

Sullivan also noted in \cite{Sullivan2006} that if Conjecture \ref{conj: path} is true
then it would imply the famous Caccetta-H\"{a}ggkvist conjecture,
which states that the minimum outdegree $h\leq \frac{n-1}{g-1}$ for any $n$-vertex oriented graph with girth $g$.
A simple proof can be found as follows.
Note that any path in an $n$-vertex oriented graph has length at most $n-1$.
If Conjecture \ref{conj: path} is true,
then there exists a path of length $h(g-1)\leq n-1$.
However, if Caccetta-H\"{a}ggkvist conjecture fails to hold, then $h>\frac{n-1}{g-1}$ and $h(g-1)> n-1$, a contradiction.

In this note,
we first present a new shorter proof of Conjecture \ref{conjecture: bermond-thomassen} for the case of $k=3$,
and then we disprove Conjecture \ref{conjecture: bang-jensen et al.} and the even girth case of Conjecture \ref{conj: path} by constructing a family of digraphs as counterexamples.

\section{Main results}

We first offer a new shorter proof of the Bermond-Thomassen conjecture for the case of $k=3$.
Note that the original proof of the corresponding result uses 10 pages in the published article \cite{LPS2009}.

\begin{theorem}\label{thm: on 3 disjoint cycles}
Every digraph with minimum outdegree at least 5 contains 3 vertex-disjoint cycles.
\end{theorem}

It is worth remarking that our proof uses some ideas appeared in \cite{LPS2009},
e.g., the ideas we used while considering the case that the minimum counterexample contains no triangle.
In fact, the main contribution of our proof is that we make the analysis part of the case that `the minimum counterexample contains a triangle' significantly shorter,
and the key purpose of this work is to motivate new ideas for checking Conjecture \ref{conjecture: bermond-thomassen} for higher values of $k$ in the coming future.

The second part of this note is devoted to Conjectures \ref{conjecture: bang-jensen et al.} and \ref{conj: path}.
We disprove Conjecture \ref{conjecture: bang-jensen et al.} by showing the following results.

\begin{theorem}\label{thm: at most k-t cycles}
For any integers $g\geq 3$ and $t\geq 1$,
there exists a digraph with girth $g$ and minimum outdegree at least $\lceil \frac{g}{g-1}k\rceil$
but contains at most $k-t$ vertex-disjoint cycles under the condition that $k\geq (t+1)(g-1)$ if $g$ is even and $k\geq (t+2)(g+1)$ if $g$ is odd.
\end{theorem}

\begin{corollary}\label{corollary: no constant c}
For any integer $g\geq 3$,
there exists no positive constant (integer) $c$ such that $f(k,g)=\lceil \frac{g}{g-1}k\rceil+c$.
\end{corollary}

For the case of $g=4$,
the conjectured result is $f(k,4)=\lceil 4k/3\rceil$,
we can also disprove Conjecture \ref{conjecture: bang-jensen et al.} by showing that $f(k,4)\geq 2k-1$.

\begin{corollary}\label{corollary: girth is 4}
For any integer $h\leq 2k-2$,
there exists a digraph with girth $4$ and minimum outdegree at least $h$ but contains no $k$ vertex-disjoint cycles.
\end{corollary}

\begin{proof}
Let $X=X_{1}\cup \cdots \cup X_{h+1}$ and $Y=\{y_{1},\ldots,y_{h+1}\}$,
where $X_{i}=\{x^{i}_{1},\ldots,x^{i}_{h}\}$ for each $1\leq i\leq h+1$ and they are disjoint.
For any bipartite digraph with bipartition $(X,Y)$,
since each cycle contains at least two vertices of $Y$ and $|Y|=h+1\leq 2k-1$,
there exist at most $k-1$ vertex-disjoint cycles.
Let $D^{*}$ be a bipartite digraph with bipartition $(X,Y)$ such that
$X_{i}$ dominates all vertices in  $Y\backslash \{y_{i}\}$ and $y_{i}$ dominates all vertices in $X_{i}$.
One can see that the minimum outdegree of $D^{*}$ is $h$ and the girth of $D^{*}$ is 4, which is a required digraph.
\end{proof}

\noindent
\textbf{Remark 1.}
Bai et al. showed in \cite{BLL2015} that every bipartite tournament (an orientation of a complete bipartite graph)
with minimum outdegree at least $2k-1$ contains $k$ vertex-disjoint cycles.
The digraph (bipartite tournament) $D^{*}$ constructed in the proof of Corollary \ref{corollary: girth is 4}
shows that the minimum outdegree $2k-1$ is best possible for Conjecture \ref{conjecture: bermond-thomassen}
even when we restrict the digraph class to bipartite tournaments.

As a final conclusion,
we disprove the even girth case $g\geq 4$ of Conjecture \ref{conj: path} and present an upper bound of $l(h,g)$ for even $g\geq 4$.
Here note that any oriented graph has girth at least 3.

\begin{theorem}\label{thm: even girth case}
Conjecture \ref{conj: path} does not hold for any even girth $g\geq 4$ and $l(h,g)\leq h(g-2)+1$ for even $g\geq 4$.
\end{theorem}

We present the proofs of Theorems \ref{thm: on 3 disjoint cycles}, \ref{thm: at most k-t cycles}, \ref{thm: even girth case} and Corollary \ref{corollary: no constant c} in the remaining sections.
Here we give some necessary definitions and notations.
For an arc $uv$ of a digraph $D$,
we write $u\rightarrow v$ and say $u$ {\em dominates} $v$ (or $v$ is {\em dominated} by $u$).
An arc $uv$ is {\em dominated} by a vertex $w$ if both $u$ and $v$ are dominated by $w$,
we also say that $w$ is a {\em dominating vertex} of the arc $uv$.
For two vertex-disjoint subsets $M$ and $N$ of $V(D)$,
we write $M\rightarrow N$ if each vertex of $M$ dominates all vertex of $N$.
For a subset $S$ of $V(D)$,
we sometimes use $S$ to denote the subdigraph of $D$ induced by the subset $S$,
and we use $N^{+}_{S}(M)$ (resp. $N^{-}_{S}(M)$) to denote the set of outneighbors (resp. inneighbors) of the vertices of $M$ in $S$.
For convenience, we write $v\rightarrow M$ for $\{v\}\rightarrow M$,
$M\rightarrow v$ for $M\rightarrow \{v\}$,
$N^{+}_{S}(v)$ for $N^{+}_{S}(\{v\})$
and $N^{-}_{S}(v)$ for $N^{-}_{S}(\{v\})$.
A digraph $D$ is {\em strongly connected} if
there is a directed path from $u$ to $v$ for any two vertices $u$ and $v$ of $D$.
We say that $D$ is {\em $k$-connected} if the removal of any set of fewer than $k$
vertices results in a strongly connected digraph.
The {\em strong connectivity} of a digraph $D$ is the minimum integer $k$ such that $D$ is $k$-connected.

\section{Proof of Theorem \ref{thm: on 3 disjoint cycles}}

Let $D$ be an $n$-vertex digraph with minimum outdegree at least 5.
Note that $n\geq 6$.
We proceed by induction on $n$.
If $n=6$,
then $D$ is a complete symmetric digraph
and the result holds by the observation that $D$ contains 3 vertex-disjoint 2-cycles.
Now assume that $n\geq 7$ and the statement holds for each digraph with order at most $n-1$.
Assume the opposite that there exist counterexamples
and we can let $D$ be a minimum counterexample,
which has minimum number of vertices and, subject to this, has minimum number of arcs.
It follows that each vertex of $D$ has outdegree exactly 5.

The following lemma will play an important role in our proof.

\begin{lemma}[\cite{LPS2009}] \label{lemma}
Let $D$ be a digraph such that each vertex has outdegree at least 3,
except at most two vertices which may have outdegree 2.
Then $D$ contains 2 vertex-disjoint cycles.
\end{lemma}

The results in the following claim are included in the proof of the main theorem in \cite {Thomassen1983} (see also \cite {LPS2009} )
and are not difficult to prove, here we give the sketch of the proofs for completeness.

\begin{claim}[\cite {LPS2009}]\label{claim: basic results}
Let $D$ be the minimum counterexample of Theorem \ref{thm: on 3 disjoint cycles} defined as above.
Then we have the following: (1) $D$ contains no 2-cycle; (2) each arc of $D$ is dominated by some vertex;
(3) the inneighborhood of each vertex of $D$ contains at least one cycle.
\end{claim}

\begin{proof}[Sketch of proof]
If there exists a 2-cycle,
then we can apply induction on the digraph obtained from $D$ by removing the 2-cycle.
If there exists an arc $xy$ not dominated by any vertex,
then we can apply induction on the digraph obtained from $D$ by removing all outgoing arcs from $x$ and then identifying $x$ and $y$.
The result (3) is a direct consequence of the result (2).
\end{proof}

Let $s$ be the strong connectivity of $D$
and let $S$ be a minimum vertex cut set of $D$.
Then $|S|=s$ and $D-S$ can be divided into two non-empty parts, say $A$ and $B$,
such that no vertex in $B$  dominates a vertex in $A$.
We show that $s\geq 3$.

If $s\leq 2$,
then, since $d^{+} _{B}(v)\geq 5-s\geq 3$ for each vertex $v\in B$,
we get that $B$ contains 2 vertex-disjoint cycles, say $C_1$ and $C_2$.
For a vertex $v\in A$,
by Claim \ref{claim: basic results} (3),
its inneighborhood $N_{D}^{-}(v)$ contains a cycle, say $C_{3}$.
Note that $N_{D}^{-}(v)\subseteq A\cup S$.
Thus $C_{1},C_{2},C_{3}$ are 3 vertex-disjoint cycles in $D$, a contradiction.

\begin{claim}\label{claim: no triangle}
The digraph $D$ contains no triangle.
\end{claim}

\begin{proof}
Assume the opposite that $C=(x_{1},x_{2},x_{3},x_{1})$ is a triangle of $D$.
If there are at most two vertices in $D\backslash C$ that dominate $C$,
then each vertex of $D\backslash C$ has outdegree at least 3, except at most two vertices which may have outdegree 2,
by Lemma \ref{lemma},
there exist 2 vertex-disjoint cycles, say $C_{1}$ and $C_{2}$, in $D\backslash C$.
It follows that $C_{1},C_{2},C$ are 3 vertex-disjoint cycles in $D$, a contradiction.
If there are three vertices, say $y_{1},y_{2},y_{3}$, that dominate $C$,
then, since the strong connectivity of $D$ is at least 3,
by Menger's theorem,
there exist three vertex-disjoint paths from $\{x_{1},x_{2},x_{3}\}$ to $\{y_{1},y_{2},y_{3}\}$.
These three paths together with a matching of appropriate arcs from $\{y_{1},y_{2},y_{3}\}$ to $\{x_{1},x_{2},x_{3}\}$ form 3 vertex-disjoint cycles in $D$, a contradiction.
\end{proof}

\begin{claim}\label{claim: outneighborhood}
The outneighborhood of each vertex of $D$ contains no cycle.
\end{claim}

\begin{proof}
Assume the opposite that there exists a vertex $v$ such that $N_{D}^{+}(v)$ contains a cycle, say $C^{+}$.
Recall that $N_{D}^{-}(v)$ contains at least one cycle and we can let $C^{-}$ be an induced cycle in $N_{D}^{-}(v)$.
Denote by $M$ the set of inneighbors of $C^{-}$ in $D\backslash C^{-}$.
Since $D$ contains no 2-cycle and no triangle,
we have $M\cap (C^{+}\cup \{v\})=\emptyset$.
Since $D$ contains no three vertex-disjoint cycles,
we have that $M$ is acyclic.
Let $x$ be a source of $M$ and let $y$ be an outneighbor of $x$ in $C^{-}$.
Since $xy$ has a dominating vertex and $x$ is a source of $M$,
the dominating vertex of $xy$ is the unique inneighbor $y^{-}$ of $y$ in $C^{-}$.
Now $y^{-}\rightarrow x$ and we consider the dominating vertex of the arc $y^{-}x$,
one can see that it must be the unique inneighbor $y^{--}$ of $y^{-}$ in $C^{-}$.
It follows that $y^{--}\rightarrow x$.
Consider the dominating vertex of the arc $y^{--}x$ and continue the above procedure,
we can get that $y\rightarrow x$ and $(x,y,x)$ is a 2-cycle, a contradiction.
\end{proof}

Let $D'$ be the spanning subdigraph of $D$ which consists of all arcs of the type $uv$
satisfying that $u$ is contained in some induced cycle $C_{v}$ of the subdigraph induced by $N_{D}^{-}(v)$.
For each arc $uv\in A(D')$,
in view of the successor $u^{+}$ of $u$ in $C_{v}$,
we have $N_{D}^{+}(u)\cap N_{D'}^{-}(v)\neq \emptyset$ as $u^{+}\in N_{D}^{+}(u)\cap N_{D'}^{-}(v)$.
This simple observation will be used in the following proof.

\begin{claim}\label{claim: D'}
The digraph $D'$ is $4$-regular.
\end{claim}

\begin{proof}
Since $D$ has no triangle and $N_{D}^{-}(v)$ contains a cycle for each vertex $v$,
we have $d_{D'}^{-}(v)\geq 4$.
It suffices to show that $d_{D'}^{+}(v)\leq 4$.
Assume the opposite that there exists a vertex $u$ with $d_{D'}^{+}(u)=5=d_{D}^{+}(u)$.
Then clearly $N_{D'}^{+}(u)=N_{D}^{+}(u)$.
Note that $N_{D}^{+}(u)\cap N_{D'}^{-}(v)\neq \emptyset$ for each arc $uv\in A(D')$.
It therefore follows that every vertex in $N_{D}^{+}(u)$ has an inneighbor in $N_{D}^{+}(u)$
and thus $N_{D}^{+}(u)$ contains a cycle, a contradiction to Claim \ref{claim: outneighborhood}.
\end{proof}

\begin{claim}\label{claim: arcs in D'}
An arc $uv\in A(D')$ if and only if $uv\in A(D)$ and $N_{D}^{+}(u)\cap N_{D'}^{-}(v)\neq \emptyset$.
\end{claim}

\begin{proof}
By the observation before Claim \ref{claim: D'},
it suffices to show that for every $uv\in A(D)$ if $N_{D}^{+}(u)\cap N_{D'}^{-}(v)\neq \emptyset$ then $uv\in A(D')$.
Assume the opposite that there exists an arc $uv\notin A(D')$ with $N_{D}^{+}(u)\cap N_{D'}^{-}(v)\neq \emptyset$.
Note that $|N_{D}^{+}(u)|=5$ and we can let $N_{D}^{+}(u)=\{v,v_{1},v_{2},v_{3},v_{4}\}$.
Since $D'$ is 4-regular and $uv\notin A(D')$,
then for each $i\in \{1,\ldots,4\}$
we have $uv_{i}\in A(D')$ and $N_{D}^{+}(u)\cap N_{D'}^{-}(v_{i})\neq \emptyset$, i.e., $v_{i}$ has an inneighbor in $N_{D}^{+}(u)$.
Now every vertex in $N_{D}^{+}(u)$ has an inneighbor in $N_{D}^{+}(u)$ and thus $N_{D}^{+}(u)$ contains a cycle, a contradiction to Claim \ref{claim: outneighborhood}.
\end{proof}

\begin{claim}\label{claim: no 4-cycle}
The digraph $D'$ contains no $4$-cycle.
\end{claim}

\begin{proof}
Assume the opposite that $C=(a,b,c,d,a)$ is a 4-cycle of $D'$.
By Lemma \ref{lemma} and Claim \ref{claim: outneighborhood},
there exist three vertices, say $x,y,z$, outside the cycle $C$ such that each vertex has exactly three outneighbors in $C$.
One can see that no vertex in $C$ has an outneighbor in $\{x,y,z\}$;
since otherwise either a 2-cycle or a triangle will appear.
One can also see that one vertex in $C$ is dominated by $\{x,y,z\}$.
Assume w.l.o.g. that $a$ is such a vertex.
By the definition of $D'$ and Claim \ref{claim: no triangle},
we have that the inneighborhood $N_{D'}^{-}(v)$ of each vertex $v$ of $D'$ induces a 4-cycle in $D$.
Note that $\{d,x,y,z\}$ forms no 4-cycle as there is no arc from $d$ to $\{x,y,z\}$.
It therefore follows that at least one arc in $\{xa,ya,za\}$ is not in $D'$
and assume w.l.o.g. that $xa\notin A(D')$.
Then, by Claim \ref{claim: arcs in D'},
we have $N_{D}^{+}(x)\cap N_{D'}^{-}(a)=\emptyset$
and thus $xd\notin A(D),~x\rightarrow \{a,b,c\}$, $xb, xc\in A(D')$ .

Now we claim that $\{y,z\}\rightarrow d$.
If not, then there exists one vertex in $\{y,z\}$ which dominates $\{b,c\}$
and assume w.l.o.g. that $y\rightarrow \{b,c\}$.
It follows that $yb, yc\in A(D')$
as $N^{+}_{D}(y)\cap N^{-}_{D'}(b)\neq \emptyset$, $N^{+}_{D}(y)\cap N^{-}_{D'}(c)\neq \emptyset$.
Now $a,x,y\in N_{D'}^{-}(b)$, $\{x,y\}\rightarrow a$,
a contradiction to the fact that $N_{D'}^{-}(b)$ induces a 4-cycle in $D$.
So $\{y,z\}\rightarrow d$.
It follows that $ya,za\in A(D')$
as $N^{+}_{D}(y)\cap N^{-}_{D'}(a)\neq \emptyset$, $N^{+}_{D}(z)\cap N^{-}_{D'}(a)\neq \emptyset$.
Now $d,y,z\in N_{D'}^{-}(a)$, $\{y,z\}\rightarrow d$,
a contradiction to the fact that $N_{D'}^{-}(a)$ induces a 4-cycle in $D$.
\end{proof}

Recall that $N_{D'}^{-}(v)$ induces a 4-cycle in $D$ for each vertex $v$ of $D'$.
Clearly, $D$ has a 4-cycle.
To finish the whole proof,
we show that $D$ has a 4-cycle which contains two consecutive arcs of $D'$.
Consider an induced 4-cycle $C^{x}=(x_{1},x_{2},x_{3},x_{4},x_{1})$ in $D$,
since $D'$ contains no 4-cycle,
there exists an arc not in $D'$ and assume w.l.o.g. that $x_{4}x_{1}$ is not in $D'$.
It follows that there exists another 4-cycle $C^{y}=(y_{1},y_{2},y_{3},y_{4},y_{1})$ induced by $N_{D'}^{-}(x_{1})$ and $C^{x}\cap C^{y}=\emptyset$.
Similarly, since $D'$ contains no 4-cycle,
we can assume w.l.o.g. that $y_{4}y_{1}$ is not in $D'$
and there exists another 4-cycle $C^{z}=(z_{1},z_{2},z_{3},z_{4},z_{1})$ induced by $N_{D'}^{-}(y_{1})$ and $C^{y}\cap C^{z}=\emptyset$.
Note that $D$ contains no three vertex-disjoint cycles.
Therefore, $C^{x}\cap C^{z}\neq \emptyset$.
If $x_{1}\in C^{x}\cap C^{z}$, then a 2-cycle will appear.
If $x_{2}\in C^{x}\cap C^{z}$, then a triangle will appear.
If $x_{4}\in C^{x}\cap C^{z}$, then since $x_{4}\rightarrow y_{1}$ and $y_{1}\rightarrow x_{1}$ we have $x_{4}x_{1}\in A(D')$, a contradiction.
Thus, $x_{3}\in C^{x}\cap C^{z}$ and $x_{3}y_{1}\in A(D')$.
Now the cycle $(x_{1},x_{2},x_{3},y_{1},x_{1})$ is a 4-cycle with two consecutive arcs $x_{3}y_{1},y_{1}x_{1}$ of $D'$.

We are now ready to finish the proof by getting a contradiction to Claim \ref{claim: no 4-cycle} that $D'$ contains a 4-cycle.
Let $C^{x'}=(x'_{1},x'_{2},x'_{3},x'_{4},x'_{1})$ be a 4-cycle in $D$ with two consecutive arcs of $D'$.
We can assume w.l.o.g. that $x'_{4}x'_{1}\notin A(D')$ and $x'_{1}x'_{2}, x'_{2}x'_{3}\in A(D')$.
Similar to the above analysis, we can get two 4-cycles $C^{y'}$ and $C^{z'}$ which are induced by the inneighborhood of $x'_{1}$ and $y'_{1}$ in $D'$, respectively.
Moreover, we can get that $C^{z'}\cap C^{x'}=x'_{3}$.
It therefore follows that $x'_{3}y'_{1}\in A(D')$ and now $(x'_{1},x'_{2},x'_{3},y'_{1},x'_{1})$ is a 4-cycle of $D'$, a contradiction.

\section{Proofs of Theorem \ref{thm: at most k-t cycles}, Corollary \ref{corollary: no constant c} and Theorem \ref{thm: even girth case}}

\begin{proof}[\bf{Proof of Theorem \ref{thm: at most k-t cycles}}]
We first consider the case that $g\geq 4$ is even.
Let $h=\lceil \frac{g}{g-1}k\rceil$ and $n=(\frac{g}{2}-1)\lceil \frac{g}{g-1}k\rceil+1$.
Note that
\begin{equation}\label{equation: n}
\begin{split}
n
&=\left(\frac{g}{2}-1\right)\left\lceil \frac{g}{g-1}k\right\rceil+1\\
&\leq \left(\frac{g-1}{2}-\frac{1}{2}\right)\left(\frac{g}{g-1}k+1\right)+1\\
&=\frac{gk}{2}+\frac{g-1}{2}-\frac{gk}{2(g-1)}-\frac{1}{2}+1\\
&=\frac{g}{2}\left(k+1-\frac{k}{g-1}\right).
\end{split}
\end{equation}
We construct a bipartite digraph $D=(X,Y;A)$
with $X=\{x_{0},x_{1},\ldots,x_{n-1}\}$
and $Y=Y_{0}\cup Y_{1}\cup \cdots \cup Y_{n-1}$,
where $Y_{i}$ consists of $\lceil \frac{g}{g-1}k\rceil$ independent vertices.
Let $x_{i}\rightarrow Y_{i}$ and $Y_{i}\rightarrow \{x_{i+1},x_{i+2},\ldots,x_{i+h}\}$ for each $0\leq i\leq n-1$,
where indices are taken modulo $n$.
One can see that each vertex of $D$ has outdegree $h$.
Recall that $g$ is even.
Since each cycle will use at least $\lceil n/h\rceil=\lceil \frac{g}{2}-1+\frac{1}{h}\rceil=g/2$ vertices of $X$ and $|X|=n$,
by Inequality (\ref{equation: n}),
there exist at most $k-t$ vertex-disjoint cycles in $D$ if we choose $k\geq (t+1)(g-1)$.
Now it suffices to show that the girth of $D$ is $g$.

Consider the digraph $D_{x}$ with
$V(D_{x})=\{x_{0},x_{1},\ldots,x_{n-1}\}$
and
$$
A(D_{x})=\{x_{i}x_{j}:~i+1\leq j\leq i+h,~0\leq i\leq n-1\}.
$$
It is not difficult to see that $D_{x}$ has a cycle of length less than $g/2$ if and only if $D$ has a cycle of length less than $g$.
Let $C_{x}$ be a shortest cycle in $D_{x}$.
Note that $C_{x}$ must consist of an arc $x^{-}_{i}x_{i}$ with $x^{-}_{i}\in N_{D_{x}}^{-}(x_{i})$ and a shortest path from $x_{i}$ to $x^{-}_{i}$ for some vertex $x_{i}$.
By symmetry, we can let $x_{i}=x_{0}$.
Recall that $N_{D_{x}}^{-}(x_{0})=\{x_{n-1},x_{n-2},\ldots,x_{n-h}\}$.
It then follows from the definition of $A(D_{x})$ that the shortest path $P$ from $x_{0}$ to $N_{D_{x}}^{-}(x_{0})$ has length at least $\lceil\frac{n-h}{h}\rceil=g/2-1$.
So $|A(C_{x})|=|A(P)|+1\geq g/2$ and the girth of $D$ is at least $g$.
Let $y_{i}$ be a vertex in $Y_{i}$.
Now
$$
(x_{0},y_{0},x_{h},y_{h},x_{2h},y_{2h},\ldots ,x_{(g/2-1)h},y_{(g/2-1)h},x_{0})
$$
is a cycle of length $g$ in $D$.
So the girth of $D$ is $g$ and the proof is complete.

Now we consider the case that $g\geq 3$ is odd.
Let $g=2r+1$ and $r\geq 1$.
Let $h=\lceil \frac{g}{g-1}k\rceil$ and $n=\frac{g-1}{2}\lceil \frac{g}{g-1}k\rceil+1=rh+1$.
We construct a digraph $D'$ with $V(D')=X'\cup Y'$,
where $X'=\{x'_{0},x'_{1},\ldots,x'_{n-1}\}$, 
$Y'=Y'_{0}\cup Y'_{1}\cup \cdots \cup Y'_{n-1}$
and $Y'_{i}$ consists of $\lceil \frac{g}{g-1}k\rceil$ independent vertices.
Let $x'_{i}\rightarrow Y'_{i}$ and $Y'_{i}\rightarrow \{x'_{i+1},x'_{i+2},\ldots,x'_{i+h}\}$ for each $0\leq i\leq n-1$,
where indices are taken modulo $n$.
One can see that each vertex of $D'$ has outdegree $h$ with only one exception that $x'_{rh}$ has outdegree $h+1$.
Denote by $D''$ the digraph obtained from $D'$ by removing the arc $x'_{rh}x'_{0}$.
Similar to the above analysis, 
we can get that the girth of $D''$ is $g+1$;
and moreover, for any cycle containing $x'_{rh}x'_{0}$ in $D'$,
since the shortest path from $x'_{0}$ to $x'_{rh}$ in $D'$ has length at least $2r$,
we can get that the girth of $D'$ is at least $2r+1$.
Let $y'_{i}$ be a vertex in $Y'_{i}$.
Now
$$
(x'_{0},y'_{0},x'_{h},y'_{h},x'_{2h},y'_{2h},\ldots ,x'_{rh},x'_{0})
$$
is a cycle of length $2r+1$ in $D'$.
So the girth of $D'$ is $g=2r+1$. 
Denote by $p'$ and $p''$ the maximum numbers of vertex-disjoint cycles in $D'$ and in $D''$, respectively.
It is not difficult to see that $p'-p''\leq 1$. 
Note that each cycle in $D''$ uses at least $\frac{g+1}{2}$ vertices in $X$ and 
$$
|X|=n=\frac{g-1}{2}\left\lceil \frac{g}{g-1}k\right\rceil+1\leq \frac{g}{2}k+\frac{g-1}{2}+1=\frac{g+1}{2}\left(k+1-\frac{k}{g+1}\right).
$$
So there exist at most $k-t-1$ vertex-disjoint cycles in $D''$ if we choose $k\geq (t+2)(g+1)$.
It follows that $D'$ has at most $k-t$ vertex-disjoint cycles if we choose $k\geq (t+2)(g+1)$.

Here note that we can construct a finite number of counterexamples to Conjecture \ref{conjecture: bang-jensen et al.},
which can be obtained from $D$ (or $D'$) by adding an arbitrary number of vertices $S$ such that each vertex in $S$ has exactly $\lceil\frac{g-1}{g}k\rceil$ outneighbors in $D$ and no inneighbors.
\end{proof}

\begin{proof}[\bf{Proof of Corollary \ref{corollary: no constant c}}]
We only show the even girth case.
For the odd girth case of the result,
one can deduce by similar arguments.
We consider a bipartite digraph $D=(X,Y;A)$ with
$X=\{x_{0},x_{1},\ldots,x_{n-1}\}$
and $Y=Y_{0}\cup Y_{1}\cup \cdots \cup Y_{n-1}$,
where each $Y_{i}$ consists of $\lceil \frac{g}{g-1}k\rceil+c$ independent vertices.
Let
$$
n=\left(\frac{g}{2}-1\right)\left(\left\lceil \frac{g}{g-1}k\right\rceil+c\right)+1,~~~h=\left\lceil \frac{g}{g-1}k\right\rceil+c.
$$
Let $x_{i}\rightarrow Y_{i}$ and $Y_{i}\rightarrow \{x_{i+1},x_{i+2},\ldots,x_{i+h}\}$ for each $0\leq i\leq n-1$,
where indices are taken modulo $n$.
Similar to the proof of Theorem \ref{thm: at most k-t cycles},
we can show that the girth of $D$ is $g$
and thus each cycle of $D$ uses at least $g/2$ vertices in $X$.
Since
\begin{equation}
\begin{split}
n
&=\left(\frac{g}{2}-1\right)\left(\left\lceil \frac{g}{g-1}k\right\rceil+c\right)+1\\
&\leq \left(\frac{g-1}{2}-\frac{1}{2}\right)\left(\frac{g}{g-1}k+c+1\right)+1\\
&=\frac{gk}{2}+\frac{(g-1)c}{2}+\frac{g-1}{2}-\frac{gk}{2(g-1)}-\frac{c}{2}-\frac{1}{2}+1\\
&=\frac{g}{2}\left(k+\frac{gc+g-2c}{g}-\frac{k}{g-1}\right),
\end{split}
\end{equation}
there exist at most $k-1$ vertex-disjoint cycles in $D$ if we choose
$$
k> \frac{(g-1)(gc+g-2c)}{g}.
$$
This completes the proof.
\end{proof}

\begin{proof}[\bf{Proof of Theorem \ref{thm: even girth case}}]
Let $g\geq 4$ be an arbitrary even integer.
Consider the digraph $D=(X,Y;A)$ constructed in the proof of Theorem \ref{thm: at most k-t cycles},
note that $D$ has girth $g$, minimum outdegree $h=\lceil \frac{g}{g-1}k\rceil$ and $|X|=n=(g/2-1)h+1$.
According to Conjecture \ref{conj: path},
if it is true, then there exists a path of length $h(g-1)$.
However, it is not difficult to see that the longest path of $D$ has length at most $2n-1$,
which implies a contradiction for any $k\geq 1$,
this is due to the fact that $k\geq 1$ implies $h\geq 2$ and
$$
2n-1=2\left(\left(\frac{g}{2}-1\right)h+1\right)-1=(g-2)h+1<(g-1)h.
$$
It therefore follows that Conjecture \ref{conj: path} does not hold for any even girth $g\geq 4$ and $l(h,g)\leq h(g-2)+1$ for even $g\geq 4$.
\end{proof}

\section{Acknowledgements}

The authors would like to thank Prof. Hongliang Lu for fruitful discussions on constructing the counterexamples of the even girth case of Conjecture \ref{conjecture: bang-jensen et al.},
and thank Prof. Guanghui Wang and the graduate student Donglei Yang for pointing out that
the digraphs we constructed above in fact can disprove the even girth case of Conjecture \ref{conj: path}.

\end{spacing}
\end{document}